\documentclass[reqno,12pt]{amsart}
\usepackage{color} 
\usepackage[pdftex]{hyperref}
\hypersetup{
    unicode=false,          
    pdftoolbar=true,        
    pdfmenubar=true,        
    pdffitwindow=false,     
    pdfstartview={FitH},    
    pdftitle={MCP Article},      
    pdfauthor={Michael Holst},   
    pdfsubject={Mathematics},    
    pdfcreator={Michael Holst},  
    pdfproducer={Michael Holst}, 
    pdfkeywords={PDE, analysis, mathematical physics}, 
    pdfnewwindow=true,      
    colorlinks=true,        
    linkcolor=red,          
    citecolor=blue,         
    filecolor=magenta,      
    urlcolor=cyan           
}
\usepackage{times}
\usepackage{amsfonts}
\usepackage{amsmath}
\usepackage{amsthm}
\usepackage{amssymb}
\usepackage{amsbsy}
\usepackage{amscd}
\usepackage{enumerate}
\usepackage{verbatim}
\usepackage{subfigure}
\newtheorem{theorem}{Theorem}[section]

\newtheorem{lemma}[theorem]{Lemma}

\newtheorem{proposition}[theorem]{Proposition}

\newtheorem{definition}[theorem]{Definition}

\newtheorem{remark}[theorem]{Remark}

\numberwithin{equation}{section}  
\definecolor{mblack}   {rgb} {0.2, 0.2, 0.2}
\definecolor{mred}     {rgb} {0.6, 0.2, 0.2}
\definecolor{mgreen}   {rgb} {0.2, 0.6, 0.2}
\definecolor{mblue}    {rgb} {0.2, 0.2, 0.6}
\definecolor{mcyan}    {rgb} {0.2, 0.6, 0.6}
\definecolor{mmagenta} {rgb} {0.6, 0.2, 0.6}
\definecolor{myellow}  {rgb} {0.6, 0.6, 0.2}

\newcommand{\leqs}{\leqslant}
\newcommand{\geqs}{\geqslant}

\newcommand{\tbar}{{\vert\kern-0.25ex\vert\kern-0.25ex\vert}}
\usepackage[color=green!20]{todonotes}

\begin{document}

\title[Determining Projections for 
       Non-Homogeneous Incompressible Fluids]
      {A Note On Determining Projections for  \\
      Non-Homogeneous Incompressible Fluids \\
}

\author[B. Faktor]{Benjamin Faktor}
\email{benjaminfaktor@gmail.com}
\address{Canyon Crest Academy\\
         5951 Village Center Loop Road\\
         San Diego, CA 92130}

\author[M. Holst]{Michael Holst}
\email{mholst@ucsd.edu}
\address{Department of Mathematics\\
         University of California San Diego\\
         La Jolla CA 92093}
\thanks{MH was supported in part by NSF DMS/FRG Award 1262982 and NSF DMS/CM Award 1620366.}

\date{\today}


\maketitle

\begin{abstract}
In this note, we consider a viscous incompressible fluid in a finite
domain in both two and three dimensions, and examine the question of
determining degrees of freedom (projections, functionals, and nodes).
Our particular interest is the case of non-constant viscosity, representing
either a fluid with viscosity that changes over time (such as an oil
that loses viscosity as it degrades), or a fluid with viscosity
varying spatially (as in the case of two-phase or multi-phase fluid models).
Our goal is to apply the determining projection framework developed
by the second author in previous work for weak solutions to the 
Navier-Stokes equations, in order to establish bounds on the number of 
determining functionals for this case, or equivalently, the dimension of a 
determining set, based on the approximation properties of an underlying 
determining projection.
The results for the case of time-varying viscosity mirror those
for weak solutions established in earlier work for constant viscosity.
The case of space-varying viscosity, treated within a single-fluid
Navier-Stokes model, is quite challenging to analyze, but we explore 
some preliminary ideas for understanding this case.
\end{abstract}


{\footnotesize
\tableofcontents
}
\vspace*{-0.5cm}

\section{Introduction}
\label{sec:intro}

In the following, we consider a viscous incompressible
fluid in $\Omega \subset \mathbb{R}^d$, where $\Omega$ is an open bounded
domain with Lipschitz continuous boundary, and where $d=2$ or $d=3$.
Given the kinematic viscosity $\nu>0$, and the vector volume force 
function $f(x,t)$ for each $x \in \Omega$ and $t \in (0,\infty)$, 
the governing Navier-Stokes equations (NSE) for the fluid velocity 
vector $u=u(x,t)$ and the scalar pressure field $p=p(x,t)$ are
given by:
\begin{align}
\frac{\partial u}{\partial t} - \nu \Delta u
    + (u \cdot \nabla) u + \nabla p 
&= f \ \ \mathrm{in} \ \ \Omega \times (0,\infty),
   \label{eqn:nv}
\\
\nabla \cdot u 
&= 0
  \ \ \mathrm{in} \ \ \Omega \times (0,\infty).
   \label{eqn:nv_bc}
\end{align}
One is also provided with initial conditions $u(0) = u_0$, as well as 
boundary conditions on $\partial \Omega \times (0,\infty)$.
Our goal in this article is to examine some questions about a concept
known as \emph{determining degrees of freedom} in the flow described
by~\eqref{eqn:nv}--\eqref{eqn:nv_bc}.
While the classical setting involves the assumption of a constant
bulk viscosity $\nu>0$, 
our particular interest here is in the case of non-constant viscosity, 
representing either a fluid with viscosity $\nu(t)$ that changes over time 
(such as an oil that loses viscosity as it degrades), or a fluid with 
viscosity $\nu(x)$ varying spatially (as in the case of two-phase or 
multi-phase fluid models), or both, represented by a viscosity $\nu(x,t)$
that changes over space and time.
We will assume that $\nu$ is everywhere positive as a function of time
and/or space, and will also assume that it \emph{a priori} satisfies 
some uniform pointwise upper and lower bounds, based on some underlying 
physical considerations.
Although we do not consider dependence of the viscosity on fluid velocity in
this work, we note that there has long been an active numerical simulation 
community that studies this case, and there is now also growing interest in
the analysis of the Navier-Stokes equations with variable viscosity;
cf.~\cite{HeLi20,HaPi15,ArSi15,AAL13,GuZh09,AMM07}.

The notion of {\em determining modes} for the NSE was
first introduced in~\cite{FoPr67} as an attempt to identify and estimate the
number of degrees of freedom in turbulent flows;
a thorough discussion of the role of determining sets in turbulence theory
can be found in~\cite{CFMT85}.
This core idea later led to the study of 
{\em Inertial Manifolds}~\cite{FST88}.
Estimates of the number of determining modes under various assumptions
have been developed sine the early 1980's; some examples
include~\cite{FMTT83,JoTi93}.
The notion of {\em determining nodes} and related concepts were introduced 
in~\cite{FoTe83,FoTe84}, followed by {\em determining volumes}
in~\cite{FoTi91,JoTi91}, and various estimates of their
number in different modeling scenarios can be found in e.g.~\cite{JoTi92,JoTi93}.
A unified framework for modes, nodes, and volumes was presented 
in~\cite{CJT95a,CJT97}, including the relationship to Inertial Manifolds.
In~\cite{Hols95b,HoTi96b}, we extended the results of~\cite{CJT95a,CJT97} to the more 
general setting of weak solutions lying in a suitably defined divergence-free solution space $V$ (see~\S\ref{sec:prelim} below).
In particular, we showed that if a 
\emph{projection operator}
$R_N \colon V \to V_N \subset L^2(\Omega)$ into a subset $V_N$
with $N = \dim(V_N) < \infty$,
satisfies an approximation inequality for $\gamma > 0$ of the form,
\begin{equation}
  \label{eqn:approximation-pre}
\| u - R_N u \|_{L^2(\Omega)} \leqs C_1 N^{-\gamma} \| u \|_{H^1(\Omega)},
\end{equation}
then the operator $R_N$ is a \emph{determining projection} for the
system~\eqref{eqn:nv}--\eqref{eqn:nv_bc} in the sense of 
Definition~\ref{def:determining_operator} below, provided $N$ is large enough.
Furthermore, in~\cite{Hols95b,HoTi96b}, we also derived explicit bounds on the 
dimension $N$ which guarantees that $R_N$ is determining, and we gave explicit 
constructions of determining projections for both smooth and weak solutions 
using ``rough'' finite element quasi-interpolants.
Our more recent article~\cite{HLT08a} generalized these results further,
to a more general family of regularized NSE and MHD models that
includes~\eqref{eqn:nv}--\eqref{eqn:nv_bc} as a special case.
This area of research has continued to generate substantial activity; 
a survey through 2009 appears in our earlier article~\cite{HLT08a},
and much more recent related activity includes~\cite{GaMe14,ZeGa15,GaGu18} and
the references therein, among many other related works that are too 
numerous to list here.

Bounds on the number of determining modes, nodes, and volumes are usually 
phrased in terms of a generalized {\em Grashof number}, defined for 
the two-dimensional NSE as:
$$
Gr = \frac{c_\rho^2 F}{\nu^2} = \frac{F}{\lambda_1 \nu^2},
$$
where $\lambda_1$ is the smallest eigenvalue of the Stokes operator
and $c_\rho=1/\sqrt{\lambda_1}$ is the related (best) Poincar\'{e} constant.
Here,
$F = \limsup_{t \rightarrow \infty} ( \int_{\Omega} |f(x,t)|^2 )^{1/2}$
if $f \in L^2(\Omega)$ for almost every~$t$, or
$F = \limsup_{t \rightarrow \infty} \sqrt{\lambda_1} \| f \|_{H^{-1}(\Omega)}$
if $f \in H^{-1}(\Omega)$ for almost every~$t$.
Due to the failure of the Sobolev embedding
$H^1 \hookrightarrow C^0$ in dimensions 2 and 3, determining node 
analysis, which was based on point-wise interpolants of the velocity,
was limited to $H^2$-regular solutions, although it was
understood that determining modes and volume elements 
made sense under weaker conditions.
To construct a general analysis framework for the case of weak 
(e.g., $H^1$-regular solutions)
solutions to~\eqref{eqn:nv}--\eqref{eqn:nv_bc},
in~\cite{HoTi96b} we introduced the notions of
{\em determining projections} and {\em determining functionals},
which we now define.
(The standard spaces $H$, $V$, and $V'$ for~\eqref{eqn:nv}--\eqref{eqn:nv_bc}
are reviewed below in~\S\ref{sec:prelim}.)
\begin{definition}[{{\bf Determining Projections for the NSE}}]
   \label{def:determining_operator}
Let $f(t), g(t) \in V'$ be any two forcing functions satisfying
\begin{equation}
    \label{eqn:forcing}
\lim_{t \rightarrow \infty} \| f(t) - g(t) \|_{V'} = 0,
\end{equation}
and let $u, v \in V$ be corresponding weak solutions 
to~\eqref{eqn:nv}--\eqref{eqn:nv_bc}.
The projection operator $R_N \colon V \to V_N \subset L^2(\Omega)$, 
~ $N = \dim(V_N) < \infty$,
is called a {\bf determining projection} for weak solutions of the
$d$-dimensional NSE if
\begin{equation}
  \label{eqn:project}
\lim_{t \rightarrow \infty} 
    \| R_N( u(t) - v(t)) \|_{L^2(\Omega)} = 0,
\end{equation}
implies that
\begin{equation}
    \label{eqn:determine}
\lim_{t \rightarrow \infty} \| u(t) - v(t) \|_H = 0.
\end{equation}
Given a basis $\{ \phi_i \}_{i=1}^N$ for the finite-dimensional space $V_N$, 
and a set of bounded linear functionals $\{ l_i \}_{i=1}^N$ from $V'$, 
a projection operator can be constructed as:
\begin{equation}
  \label{eqn:projectOper}
R_N u = \sum_{i=1}^N l_i(u) \phi_i.
\end{equation}
Condition \eqref{eqn:project} is then implied by:
\begin{equation}
  \label{eqn:linfunc}
\lim_{t \rightarrow \infty} 
    | l_i( u(t) - v(t)) | = 0, \ \ \ \ \ i=1, \ldots, N
\end{equation}
and in this case we refer to $\{l_i\}_{i=1}^N$ 
as a set of {\bf determining functionals}.
\end{definition}
The analysis of whether $R_N$ or $\{l_i\}_{i=1}^N$ are determining 
can be reduced to an analysis of the approximation properties of $R_N$.
Note that in this construction,
the basis $\{\phi_i\}_{i=1}^N$ need not span a subspace of the solution 
space $V$, so that the functions $\phi_i$ need not, for example,
be divergence-free.
Note that Definition~\ref{def:determining_operator}
encompasses each of the notions of determining modes, nodes, and volumes
by making particular choices for the sets of functions 
$\{\phi_i\}_{i=1}^N$ and $\{l_i\}_{i=1}^N$.

\subsection*{Outline}

Preliminary material is presented in~\S\ref{sec:prelim},
including notation used for Lebesgue and Sobolev spaces and norms,
and some inequalities for bounding the terms appearing 
in weak formulations of the NSE.
In~\S\ref{sec:approx-theory}, we given an overview of the 
general framework for constructing determining projections for the NSE
for both two and three spatial dimensions.
To make use of the framework to establish bounds on the number
of determining degrees of freedom for weak solutions, one must assume, 
or establish, a single \emph{a priori} bound for solutions to the equations,
and also provide a projection operator that satisfies a single 
approximation inequality.
The remainder of the article then turns to the necessary \emph{a priori}
bounds for non-constant viscosity.
In~\S\ref{sec:time-varying-estimates}, 
we derive some \emph{a priori} bounds for the case of
time-varying viscosity that are needed to make use of the determining 
projection framework in~\S\ref{sec:approx-theory}.
Section~\S\ref{sec:space-varying}
looks briefly at a simplified model for space-varying viscosity.
We first develop a natural weak formulation for a simplified model, 
where the viscosity is allowed to now be space-varying, but is also assumed 
to be explicitly given as data, and in particular, does not depend on the 
fluid velocity.
Using this simplified formulation, we then establish some basic 
\emph{a priori} bounds for use with the 
determining projection framework from~\S\ref{sec:approx-theory}.
Some additional technical tools are summarized in Appendix~\ref{app:technical},
\emph{a priori} estimates for the constant viscosity case
are presented in Appendix~\ref{app:apriori-constant-viscosity}.

\section{Preliminary material}
\label{sec:prelim}

We briefly review some background material and notation,
following the approach taken in our earlier 
articles~\cite{Hols95b,HoTi96b,HLT08a},
which in turn followed the notational 
conventions used in~\cite{CoFo88,Lion69,Tema77,Tema83}.
To keep the discussion in this section as clear and concise as possible, 
we have placed some technical results that are repeatedly used throughout
the paper in Appendix~\ref{app:technical}.

Let $\Omega \subset \mathbb{R}^d$ denote an open bounded set.
The embedding and other standard results we will need to rely on are known 
to hold for example if the domain $\Omega$ has a locally Lipschitz boundary, 
denoted as $\Omega \in \mathcal{C}^{0,1}$ (cf.~\cite{Adam78}).
For example, open bounded convex sets $\Omega \subset \mathbb{R}^d$ 
satisfy $\Omega \in \mathcal{C}^{0,1}$ (Corollary~1.2.2.3 in~\cite{Gris85}),
so that convex polyhedral domains
are in $\mathcal{C}^{0,1}$.
Let $H^k(\Omega)$ denote the usual Sobolev spaces $W^{k,2}(\Omega)$.
Employing multi-index notation, the distributional partial derivative of 
order $|\alpha|$ is denoted $D^{\alpha}$, so that the 
(integer-order) norms and semi-norms in $H^k(\Omega)$ may be denoted
$$
\| u \|^2_{H^k(\Omega)}
    = \sum_{j=0}^k |\Omega|^{\frac{j-k}{d}} | u |_{H^j(\Omega)}^2,
\ \ \ \ \ \ 
| u |^2_{H^j(\Omega)} 
    = \sum_{| \alpha | = j} \| D^{\alpha} u \|_{L^2(\Omega)},
\ \ \ 
0 \leqs j \leqs k,
$$
where $|\Omega|$ represents the measure of $\Omega$.
Fractional order Sobolev spaces and norms may be defined for example
through Fourier transform and extension theorems, or through 
interpolation.
A fundamentally important subspace is the $k=1$ case of
$$
H^k_0(\Omega) = \{ \mathrm{closure~of}~ C_0^{\infty}(\Omega)
        ~\mathrm{in}~ H^k(\Omega) \},
$$
for which the Poincar\'{e} Inequality holds.
(See Lemma~\ref{app:lemma:poincare} in Appendix~\ref{app:technical}.)
The spaces above extend naturally (cf.~\cite{Tema77}) to product spaces of 
vector functions $u=(u_1,u_2,\ldots,u_d)$, which are denoted with the
same letters but in bold-face; for example,
$\mathbf{H}^k_0(\Omega) = \left( H^k_0(\Omega) \right)^d$.
The inner-products and norms in these product spaces are extended in the 
natural Euclidean way; the convention here will be to subscript these 
extended vector norms the same as the scalar case.

Define now the space $\mathcal{V}$ of divergence free 
$\mathbf{C}^{\infty}$ vectors with compact support as
$$
\mathcal{V} = \left\{ \phi \in \mathbf{C}_0^{\infty}(\Omega)
            ~|~ \nabla \cdot \phi = 0 \right\}.
$$
Two subspaces of $\mathbf{L}^2(\Omega)$ and $\mathbf{H}^1_0(\Omega)$
are fundamental to the study of the NSE:
$$
H = \{ \mathrm{closure~of}~ \mathcal{V} ~\mathrm{in}~ \mathbf{L}^2(\Omega) \},
\ \ \ \ \ \ \ \ \ \ 
V = \{ \mathrm{closure~of}~ \mathcal{V} ~\mathrm{in}~ \mathbf{H}^1_0(\Omega) \}.
$$
We use a fairly standard notation (cf.~\cite{Tema77}) 
for inner-products and norms in $H$ and $V$ is:
\begin{align}
(u,v)_H &= (u,v)_{L^2(\Omega)},
\quad
\|u\|_H = \|u\|_{L^2(\Omega)},
   \label{eqn:notation-H}
\\
(u,v)_V &= [u,v]_{H^1(\Omega)},
\quad
\|u\|_V = |u|_{H^1(\Omega)}.
   \label{eqn:notation-V}
\end{align}
Thanks to the Poincar\'{e} inequality,
the $H^1$-semi-inner-product $[u,v]_{H^1(\Omega)}$ is an
inner-product on $V$, and the
$H^1$-semi-norm $|u|_{H^1(\Omega)}$ is a norm on $V$.

The NSE~\eqref{eqn:nv}--\eqref{eqn:nv_bc} with homogeneous Dirichlet
(no-slip) boundary conditions are
equivalent (cf.~\cite{Tema77}) to the functional differential equation:
\begin{equation}
\frac{du}{dt} + \nu Au + B(u,u) = f, \ \ \ \ \ u(0) = u_0,
   \label{eqn:strong}
\end{equation}
where the Stokes operator $A$ and bilinear map $B$ are defined as
$$
Au = -P \Delta u,
\ \ \ \ \ B(u,v) = P [(u \cdot \nabla) v],
$$
where the operator $P$ is the Leray orthogonal projector,
$P \colon H_0^1 \to V$ and $P \colon L^2 \to H$, respectively.
Weak formulations of the NSE will use the 
bilinear Dirichlet form $a(\cdot,\cdot)$ and trilinear form 
$b(\cdot,\cdot,\cdot)$ as:
$$
a(u,v) = (\nabla u, \nabla v)_H,
\ \ \ \ \ b(u,v,w) = (B(u,v),w)_H = (P ((u \cdot \nabla) v), w)_H.
$$
Again, thanks to the Poincar\'{e} inequality,
the form $a(\cdot,\cdot)$ is
actually an inner-product on V, and as noted above, the induced semi-norm 
$|\cdot|_{H^1(\Omega)}=a(\cdot,\cdot)^{1/2}$ is in fact a norm on V, 
equivalent to the $H^1$-norm.
{\em A priori} bounds and various symmetries can be derived for the 
trilinear form $b(\cdot,\cdot,\cdot)$;
the results of this type that we will need are collected together
in Appendix~\ref{app:technical}.

A general weak formulation of the NSE~\eqref{eqn:nv}--\eqref{eqn:nv_bc} can be
written as~(cf.~\cite{Tema77,Tema83,CoFo88}):
\begin{definition}[{{\bf Weak Solutions of the NSE}}]
Given $f \in L^2([0,T];V')$, a weak solution of the NSE
satisfies $u \in L^2([0,T];V) \cap C_w([0,T];H)$, 
$du/dt \in L^1_{\text{loc}}((0,T];V')$, and
\begin{align}
< \frac{du}{dt} ,v >
    + \nu a(u,v) + b(u,u,v) &= <f,v>, 
      \ \ \forall v \in V, 
      \ \ \text{~for~almost~every~} t,
   \label{eqn:weak}
\\
u(0) &= u_0.
   \label{eqn:weak_ic}
\end{align}
\end{definition}
Here, the space $C_w([0,T];H)$ is the subspace of $L^{\infty}([0,T];H)$ of
weakly continuous functions, and $<\cdot,\cdot>$ denotes the duality pairing
between $V$ and $V'$, where $H$ is the Riesz-identified pivot space in the
Gelfand triple $V \subset H=H' \subset V'$.
Note that since the Stokes operator can be uniquely extended 
to $A \colon V \to V'$, and since it can be shown that
$B \colon V \times V \to V'$ (cf.~\cite{CoFo88,Tema83} for both results),
the functional form~\eqref{eqn:strong} still makes sense for weak 
solutions, and the total operator represents a mapping $V \to V'$.

Consider now two forcing functions $f, g \in L^2([0,\infty];V')$ and
corresponding weak solutions $u$ and $v$ to~\eqref{eqn:strong} 
in either the two- or three-dimensional case.
Subtracting equations~\eqref{eqn:strong} for $u$ and $v$ yields
an equation for the difference function $w=u-v$, namely
\begin{align}
\frac{dw}{dt} + \nu A w + B(u,u) - B(v,v) &= f-g.
  \label{eqn:above}
\end{align}
Since the residual of~\eqref{eqn:above} lies in the dual space $V'$,
for almost every $t$, 
we can consider duality pairing of~\eqref{eqn:above} with a function 
in $V$, and in particular with $w \in V$, which yields
$$
<\frac{dw}{dt}, w>
    + \nu a(w,w) + b(u,u,w) - b(v,v,w) = <f-g,w> \ \text{~for~almost~every~} t.
$$
Using the notation~\eqref{eqn:notation-H}--\eqref{eqn:notation-V}
going forward, it can be shown (cf.~\cite{Tema77,Tema83}) that
\begin{align}
\frac{1}{2} \frac{d}{dt} \|w\|_H^2 &= <\frac{dw}{dt}, w>
\label{eqn:dbydt}
\end{align}
in the sense of distributions.
Lemma~\ref{app:lemma:symmetries} in Appendix~\ref{app:technical}
establishes the symmetry relation 
$b(w,u,w) = b(u,u,w) - b(v,v,w), \forall u,v,w \in V$, 
so the function $w=u-v$ must satisfy
\begin{align}
\frac{1}{2}\frac{d}{dt} \|w\|_H
   + \nu \|w\|_V^2 + b(w,u,w) &= <f-g,w>.
   \label{eqn:weak_diff}
\end{align}
Equation~\eqref{eqn:weak_diff} will be the starting point for our analysis of determining projections below. 
In the introduction, we highlighted an approximation property~\eqref{eqn:approximation-pre} that  we will assume that a determining projection satisfies, and we will give an explicit example of such a projection below from~\cite{Hols95b,HoTi96b}.
A useful consequence of property~\eqref{eqn:approximation-pre} that was noted in~\cite{HoTi96b} is the following.
\begin{lemma}
  \label{lemma:approximation}
Let $R_N \colon V \to V_N \subset L^2(\Omega)$, ~ $N = \dim(V_N) < \infty$,
satisfy the following approximation inequality for some $\gamma > 0$:
\begin{align}
\| u - R_N u \|_{L^2(\Omega)} & \leqs C_1 N^{-\gamma} \| u \|_{H^1(\Omega)}.
  \label{eqn:approximation}
\end{align}
Then the following inequalities hold:
\begin{align}
\|u\|_{L^2(\Omega)} 
  & \leqs 2 C_1^2 N^{-2\gamma} \|u\|_{H^1(\Omega)}
  + 2 \|R_N u\|_{L^2(\Omega)}^2,
   \label{eqn:approximation-inequality-1}
\\
\|u\|_{H^1(\Omega)} 
  & \geqs [N^{2\gamma} / (2 C_1^2)] \|u\|_{L^2(\Omega)}
  - [N^{2\gamma} /C_1^2] \|R_N u\|_{L^2(\Omega)}^2.
   \label{eqn:approximation-inequality-2}
\end{align}
\end{lemma}
\begin{proof}
We start with squaring~\eqref{eqn:approximation},
\begin{align*}
\| u - R_N u \|_{L^2(\Omega)}^2 
   & = \| u \|_{L^2(\Omega)}^2 - 2(u,R_N u)_{L^2(\Omega)}
         + \| R_N u \|_{L^2(\Omega)}^2
   \leqs C_1^2 N^{-2\gamma} \| u \|_{H^1(\Omega)}^2.
\end{align*}
Rearranging the inequality we have
\begin{align*}
\| u \|_{L^2(\Omega)}^2 
& \leqs C_1^2 N^{-2\gamma} \| u \|_{H^1(\Omega)}^2
  - \| R_N u \|_{L^2(\Omega)}^2
  + 2(u,R_N u)_{L^2(\Omega)}
\\
& \leqs C_1^2 N^{-2\gamma} \| u \|_{H^1(\Omega)}^2
  - \| R_N u \|_{L^2(\Omega)}^2
  + 2\|u\|_{L^2(\Omega)} \|R_N u\|_{L^2(\Omega)}
\\
 & =     C_1^2 N^{-2\gamma} \| u \|_{H^1(\Omega)}^2
   - \| R_N u \|_{L^2(\Omega)}^2
   + \left(\tfrac{1}{\sqrt{2}} \|u\|_{L^2(\Omega)} \right)
     \left(2 \sqrt{2} \|R_N u\|_{L^2(\Omega)} \right)
\\
 & \leqs C_1^2 N^{-2\gamma} \| u \|_{H^1(\Omega)}^2
   - \| R_N u \|_{L^2(\Omega)}^2
   + \tfrac{1}{2}\|u\|_{L^2(\Omega)}^2
   + 2 \|R_N u\|_{L^2(\Omega)}^2,
\end{align*}
which after multiplying through by 2 and simplifying 
gives~\eqref{eqn:approximation-inequality-1}.
Rearrangement of the terms in~\eqref{eqn:approximation-inequality-1}
then also gives~\eqref{eqn:approximation-inequality-2}
\end{proof}
Finally, we note key tools for establishing a number of \emph{a priori} estimates in Sections~\ref{sec:time-varying-estimates} and~\ref{sec:space-varying} will be both classical and generalized Gronwall inequalities (cf.~\cite{FMTT83,JoTi91}) which we include in Appendix~\ref{app:technical}.

\section{The framework for constructing determining projections}
\label{sec:approx-theory}

We now give an overview of the general framework for constructing determining 
projections for the NSE for both two and three spatial dimensions,
represented by Theorem~\ref{thm:main-2d} below.
To make use of the framework to establish bounds on the number
of determining degrees of freedom for weak solutions, one must assume, 
or establish, a single \emph{a priori} bound for solutions to the equations,
(inequality~\eqref{eqn:energy-for-thm} below)
and also provide a projection operator that satisfies a single 
approximation inequality
(inequality~\eqref{eqn:approximation} above).

Our earlier results in~\cite{Hols95b,HoTi96b} are included as 
particular instances of this framework, and
we include in Appendix~\ref{app:apriori-constant-viscosity}
the well-known \emph{a priori} bounds
for constant viscosity in the $d=2$ and $d=3$ cases that were used 
in~\cite{Hols95b,HoTi96b,HLT08a}.
The remainder of the article then turns to the necessary \emph{a priori}
bounds for non-constant viscosity.

\begin{theorem}[{{\bf Existence of Determining Projections for the NSE on domains $\Omega \subset \mathbb{R}^2$}}]
   \label{thm:main-2d}
Let $f(t), g(t) \in V'$ be any two forcing functions satisfying
$$
\lim_{t \rightarrow \infty} \| f(t) - g(t) \|_{V'} = 0,
$$
and let $u, v \in V$ be the corresponding weak solutions 
to~\eqref{eqn:nv}--\eqref{eqn:nv_bc} for $d=2$.
If there exists a projection operator 
$R_N \colon V \to V_N \subset L^2(\Omega)$, $N = \dim(V_N)$,
satisfying
$$
\lim_{t \rightarrow \infty} 
    \| R_N( u(t) - v(t)) \|_{L^2(\Omega)} = 0,
$$
and satisfying for $\gamma>0$ the approximation inequality
$$
\| u - R_N u \|_{L^2(\Omega)} \leqs C_1 N^{-\gamma} \| u \|_{H^1(\Omega)},
$$
then
$$
\lim_{t \rightarrow \infty} \| u(t) - v(t) \|_H = 0
$$
holds if $N$ is such that
$$
\infty > N > C \left( \frac{1}{\nu^2} 
     \limsup_{t \rightarrow \infty} \|f(t)\|_{V'}
     \right)^{\frac{1}{\gamma}},
$$
where $C$ is a constant independent of $\nu$ and $f$.
\end{theorem}
\begin{proof}
Staying with the notation~\eqref{eqn:notation-H}--\eqref{eqn:notation-V}, 
we begin with equation~\eqref{eqn:weak_diff}, 
employing inequality~\eqref{app:eqn:lady_2d}
from Theorem~\ref{app:lemma:lady} in Appendix~\ref{app:technical},
along with Cauchy-Schwarz and Young's inequalities, to yield 
\begin{align*}
\frac{1}{2} \frac{d}{dt} \|w\|_H^2 + \nu \| w \|_V^2
& \leqs \|u\|_V |w|_H \| w \|_V
    + \|f-g\|_{V'} \|w\|_V
\\
& \leqs \frac{1}{\nu} \|u\|_V^2 \|w\|_H^2
  + \frac{1}{\nu} \|f-g\|_{V'}^2
  + \frac{\nu}{2} \|w\|_V^2.
\end{align*}
Equivalently, this is
$$
\frac{d}{dt} \|w\|_H^2 + \nu \| w \|_V^2
 -  \frac{2}{\nu} \|u\|_V^2 \|w\|_H^2
\leqs \frac{2}{\nu} \|f-g\|_{V'}^2.
$$
To bound the second term on the left from below, we employ 
the approximation inequality~\eqref{eqn:approximation-inequality-2}
from Lemma~\ref{lemma:approximation},
which yields
$$
\frac{d}{dt} \|w\|_H^2 
  +  \left( \frac{\nu N^{2\gamma}}{2 C_1^2} 
  -   \frac{2}{\nu} \|u\|_V^2 \right)
      \|w\|_H^2 
\leqs \frac{2}{\nu} \|f-g\|_{V'}^2
+ \frac{\nu N^{2\gamma}}{C_1^2} \|R_N w\|_{L^2(\Omega)}^2.
$$
This is a differential inequality of the form
$$
\frac{d}{dt} \|w\|_H^2 + \alpha \|w\|_H^2 \leqs \beta,
$$
with obvious definition of $\alpha$ and $\beta$.

The Generalized Gronwall Lemma~\ref{app:lemma:gronwall-generalized} 
can now be applied.
Recall that we have assumed both $\|f-g\|_{V'} \rightarrow 0$ and
$\|R_N w\|_{L^2(\Omega)} \rightarrow 0$ 
as $t \rightarrow \infty$.
Since it is assumed that $u$ and $v$, and hence $w$, are in $V$,
so that all other terms appearing in $\alpha$ and $\beta$ remain bounded,
it must hold that
$$
\lim_{t \rightarrow \infty}
     \frac{1}{T} \int_{t}^{t+T} \beta^+(\tau) d\tau = 0,
\ \ \ \ \ \ \ \ 
\limsup_{t \rightarrow \infty}
    \frac{1}{T} \int_{t}^{t+T} \alpha^-(\tau) d\tau <\infty.
$$
It remains to verify that for some fixed $T>0$,
$$
\limsup_{t \rightarrow \infty}
    \frac{1}{T} \int_{t}^{t+T} \alpha(\tau) d\tau >0.
$$
This means we must verify the following inequality for some fixed $T > 0$:
\begin{align}
   \label{eqn:key}
N^{2\gamma}
&> 
\frac{2 C_1^2}{\nu} \left( \limsup_{t \rightarrow \infty}
\frac{1}{T} \int_{t}^{t+T} \frac{2 \|u\|_V^2}{\nu} d\tau
\right)
= \frac{4 C_1^2}{\nu^2} 
\limsup_{t \rightarrow \infty} \frac{1}{T}
      \int_{t}^{t+T} \|u\|_V^2 d\tau.
\end{align}
The following {\em a priori} bound on any weak solution 
can be shown to hold 
(Lemma~\ref{app:lemma:energy2} in 
Appendix~\ref{app:apriori-constant-viscosity}):
\begin{align}
\limsup_{t \rightarrow \infty} \frac{1}{T} \int_{t}^{t+T} \|u(\tau)\|_H^2 d\tau 
& \leqs \frac{2}{\nu^2} \limsup_{t \rightarrow \infty} \|f(t)\|_{V'}^2,
\label{eqn:energy-for-thm}
\end{align}
for any $T > c_{\rho}^2/\nu > 0$,
where $c_{\rho}$ is the best constant from the Poincar\'{e}
inequality (Lemma~\ref{app:lemma:poincare} in Appendix~\ref{app:technical}).
Therefore, if
\begin{align}
N^{2\gamma}
&> 8 C_1^2 \left( 
  \frac{1}{\nu^2} 
   \limsup_{t \rightarrow \infty} \|f(t)\|_{V'}
\right)^2
\geqs \frac{4 C_1^2}{\nu^2}
   \left( \frac{2}{\nu^2}
  \limsup_{t \rightarrow \infty} \|f(t)\|_{V'}^2
\right),
\end{align}
implying that~\eqref{eqn:key} holds, then by the Gronwall 
Lemma~\ref{app:lemma:gronwall-generalized}, it follows that
$$
\lim_{t \rightarrow \infty} \|w(t)\|_H
   = \lim_{t \rightarrow \infty} \| u(t) - v(t) \|_H = 0.
$$
\end{proof}
\begin{remark}
Theorem~\ref{thm:main-2d} for the $d=2$ case can be extended to $d=3$ in several ways, which we will not reproduce here.
For example, in~\cite{HoTi96b} a finite energy dissipation assumption was used to extend Theorem~\ref{thm:main-2d} to $d=3$ case; a different approach for the $d=3$ case is taken in~\cite{Hols95b}.
\end{remark}
\begin{remark}
In the case of constant viscosity, the required \emph{a priori}
estimate in~\eqref{eqn:energy-for-thm} in the proof of 
Theorem~\ref{thm:main-2d} is provided by 
Lemma~\ref{app:lemma:energy2} in 
Appendix~\ref{app:apriori-constant-viscosity}).
For the case of time-varying viscosity, the required estimate is provided by
Lemma~\ref{lemma:energy2-time} or
Lemma~\ref{lemma:energy3-time}
in Section~\ref{sec:time-varying-estimates}.
For the case of our simple model of space-varying viscosity, 
the required estimate is provided by
Proposition~\ref{proposition:energy2-space}
in Section~\ref{sec:space-varying}.
\end{remark}

\section{A priori estimates for time-varying viscosity}
\label{sec:time-varying-estimates}

In this section, we develop the \emph{a priori} estimates needed to apply the 
determining projection framework from~\S\ref{sec:approx-theory}
(Theorem~\ref{thm:main-2d}) to the case of a time-varying viscosity function 
$\nu = \nu(t)$. 
The first is an $L^2$ estimate that is used to prove the second and third estimates, following the strategy for the case of constant viscosity (see Appendix~\ref{app:apriori-constant-viscosity}).
The second estimate is what is needed for use with with Theorem~\ref{thm:main-2d} in different contexts.
The time-varying viscosity is assumed to satisfy $\nu \in L^1(0,T)$
and obey the \emph{a priori} pointwise bounds:
\begin{align}
\label{eqn:time-nu-1}
0 < \underline{\nu} & \leqs \nu(t) \leqs \overline{\nu} < +\infty,
\quad \forall t \in (0,T),
\end{align}
where
\begin{align}
\label{eqn:time-nu-2}
\underline{\nu} = \inf_{0 < t < T} \nu(t),
&
\qquad
\overline{\nu}  = \sup_{0 < t < T} \nu(t).
\end{align}

The first estimate gives a bound on the $L^2$-norm of a weak
solution to~\eqref{eqn:nv}--\eqref{eqn:nv_bc}.
\begin{lemma}[{{\bf $L^2$-Estimates, time-varying viscosity}}]
   \label{lemma:energy1-time}
Let $\nu \in L^1(0,T)$ and assume
that~\eqref{eqn:time-nu-1}--\eqref{eqn:time-nu-2} hold.
Let $u \in L^2((0,T);V)$ be a weak solution of the Navier-Stokes 
equations~\eqref{eqn:nv}--\eqref{eqn:nv_bc}, 
with Lipschitz domain $\Omega \subset \mathbb{R}^d$, $d=2$ or $d=3$.
It holds that
\begin{equation}    \label{eqn:time-energy1}
       \limsup_{t\rightarrow \infty}
       \|u(t) \|_{H}^2 
       \leqs \frac{\bar{K}}{\underline{\nu}}
       \limsup_{t \rightarrow \infty}
       \|f(t) \|_{V'}^2 
\end{equation}
where \(K = \limsup_{t \rightarrow \infty} \int_{0}^{t} e^{-\phi_{s}(t)) / c_\rho^2} ds\), \(\phi_{s}(t) = \int_{s}^t \nu(z)dz \), and \(c_{\rho}\) is the Poincar\'{e} constant.
\end{lemma}
\begin{proof}
Beginning with equation \eqref{eqn:weak} for $v=u$, 
using~\eqref{eqn:dbydt}, and noting 
that Theorem~\ref{app:lemma:symmetries} in Appendix~\ref{app:technical}
ensures $b(u,u,u)=0$, we are left with
\begin{align*}
\frac{1}{2} \frac{d}{dt}\|u\|_H^2
   + \nu \| u \|_V^2
& \leqs \|f\|_{V'} \|u\|_V.
\end{align*}
Applying Young's inequality leads to
\begin{align*}
\frac{d}{dt}\|u\|_H^2
   + 2 \nu \| u \|_V^2 
 & \leqs \left( \sqrt{\frac{2}{\nu}} \|f\|_{V'} \right)
         \left( \sqrt{2 \nu} \|u\|_{V} \right)
  \leqs \frac{1}{\nu} \|f\|_{V'}^2
    + \nu \|u\|_{V}^2,
\end{align*}
which gives then
\begin{align}
\frac{d}{dt}\|u\|_H^2
  + \nu \|u\|_{V}^2 
& \leqs \frac{1}{\nu} \|f\|_{V'}^2.
   \label{eqn:time-weak-start1}
\end{align}
Employing the Poincar\'{e} inequality we end up with
\begin{align*}
\frac{d}{dt}\|u\|_H^2
  + \frac{\nu}{c_{\rho}^2} \|u\|_H^2 
& \leqs \frac{1}{\nu} \|f\|_{V'}^2.
\end{align*}
This is a differential inequality for $\|u(t)\|_H^2$,
and by Gronwall's Inequality (Lemma~\ref{app:lemma:gronwall-classical}) 
it holds that
\begin{align*}
\|u(t)\|_H^2 
& \leqs 
   \|u(r)\|_H^2 e^{-\int_r^t \nu(\tau)/c_{\rho}^2 d\tau}
    + \int_r^t \frac{1}{\nu(s)} \|f(s)\|_{V'}^2
        e^{-\int_{s}^t \nu(\tau)/c_{\rho}^2 d\tau} ds
\\
& = \|u(r)\|_H^2 e^{-1/c_\rho^2 \int_{s}^{t} \nu(\tau) d\tau}
    + \int_r^t \frac{1}{\nu} e^{-1/c_\rho^2 \int_{s}^{t} \nu(\tau) d\tau}
        \|f(s)\|_{V'}^2 ds
\\
&   =  \|u(r)\|_H^2 e^{-\phi_r(t)/c_{\rho}^2}
    + \int_r^t \frac{1}{\nu} e^{-\phi_s(t)/c_{\rho}^2}
        \|f(s)\|_{V'}^2 ds
\\
&   \leqs  \|u(r)\|_H^2 e^{-\phi_r(t)/c_{\rho}^2}
    + \sup_{r\leqs \delta \leqs t} \frac{1}{\nu(\delta)}\|f(\delta)\|_{V'}^2 
\int_r^t e^{-\phi_s(t)/c_{\rho}^2} ds
\\
&    \leqs  \|u(r)\|_H^2 e^{\phi_r(t)/c_{\rho}^2}
    + \frac{K}{\underline{\nu}} \sup_{r\leqs \delta \leqs t} \|f(\delta)\|_{V'}^2 
 ,
\end{align*}
where \(\phi_c(t) = \int_c^t \nu(z)dz\) and \(K = \int_0^t e^{-\phi_s(t)/c_\rho^2}ds\),
which must hold for every $r \in (0,t]$.
Taking the $\limsup_{t\rightarrow \infty}$ of both
sides of the inequality leaves~\eqref{eqn:time-energy1} with obvious definition of \(\bar{K}\).
\end{proof}


The following estimate gives the a bound on the time-averaged $H^1$-semi-norm 
of a weak solution to~\eqref{eqn:nv}--\eqref{eqn:nv_bc}.
\begin{lemma}[{{\bf Time-averaged $H^1$-Estimates, time-varying viscosity}}]
   \label{lemma:energy3-time}
Let \\ $u \in L^2((0,T);V)$ be a weak solution of the Navier-Stokes 
equations ~\eqref{eqn:nv}--\eqref{eqn:nv_bc},
with Lipschitz domain $\Omega \subset \mathbb{R}^d$, $d=2$ or $d=3$.
Then for every $T$ with $T \geqs c_{\rho}^2/\underline{\nu} > 0$ it holds that
\begin{equation}
\label{eqn:time-energy3}
       \limsup_{t \rightarrow \infty}
       \frac{1}{T}
       \int_{t}^{t+T}
       \|u(\tau)\|_V^2
       d\tau
\leqs
       \frac{\bar{K}\underline{\nu}^3 + c_{\rho}^2}{\underline{\nu}^2 c_{\rho}^2}
       \limsup_{t \rightarrow \infty}
       \|f(t) \|_{V'}^2
\end{equation}
where \(\bar{K} = \limsup_{t \rightarrow \infty} \int_{0}^{t} e^{-\phi_{s}(t)) / c_\rho^2} ds\), \(\phi_{s}(t) = \int_{s}^t \nu(z)dz \), and \(c_{\rho}\) is the Poincar\'{e} constant.
\end{lemma}
\begin{proof}
We begin with~\eqref{eqn:time-weak-start1}, which was
\begin{align*}
\frac{d}{dt}\|u\|_H^2
  + \nu \|u\|_{V}^2 
& \leqs \frac{1}{\nu} \|f\|_{V'}^2.
\end{align*}
Dividing by \(\nu(t)\) and integrating from $t$ to $t+T$ with $T>0$ gives
\begin{align*}
\int_{t}^{t+T} \frac{1}{\nu(t)} \frac{d}{dt} \|u(\tau)\|_H^2 d\tau
   + \int_t^{t+T} \|u(\tau)\|_V^2 d\tau 
& \leqs  \int_t^{t+T} \frac{1}{\nu^2(\tau)}\|f(\tau)\|_{V'}^2 d\tau.
\end{align*}
Integrating by parts the left-most term gives
\[
\frac{1}{\nu(t+T)} \|u(t+T)\|_H^2
- \frac{1}{\nu(t)} \|u(t)\|_H^2
 + \int_t^{t+T} \frac{1}{\nu^2(\tau)} \|u(\tau)\|_H^2 d\tau
 \]
 \[
+ \int_t^{t+T} \|u(\tau)\|_V^2 d\tau
\leq \int_{t}^{t+T} \frac{1}{\nu^2(\tau)} \|f(\tau)\|_{V'}^2 d\tau
\]
Dropping the positive first and third terms on the left and bounding the integral
on the right gives
\begin{align*}
\int_t^{t+T} \|u(\tau)\|_V^2 d\tau 
& \leqs \frac{1}{\nu(t)}\|u(t)\|_H^2
  + T \sup_{t\leqs s \leqs t+T} \frac{1}{\nu^2(s)}\|f(s)\|_{V'}^2.
\end{align*}
Taking the $\limsup_{t \rightarrow \infty}$ of both sides, and
dividing by $T$, gives
\begin{align*}
\limsup_{t \rightarrow \infty}
\frac{1}{T}
\int_t^{t+T}\|u(\tau)\|_V^2 d\tau 
& \leqs \frac{1}{ T} 
  \limsup_{t \rightarrow \infty} \frac{1}{\nu(t)}\|u(t)\|_H^2
  + \limsup_{t\rightarrow \infty}
          \frac{1}{\nu^2(t)}\|f(t)\|_{V'}^2
 \\
  & \leqs \frac{1}{\underline{\nu}T}
  	   \limsup_{t\rightarrow \infty} \|u(t)\|_H^2
       + \frac{1}{\underline{\nu}^2}
       \limsup_{t\rightarrow \infty} \|f(t)\|_{V'}^2.
\end{align*}
Using the estimate from Lemma~\ref{lemma:energy1-time} and bounding the right-most term gives then 
\begin{align*}
\limsup_{t \rightarrow \infty} \frac{1}{T}
\int_t^{t+T} \|u(\tau)\|_V^2 d\tau 
& \leqs \left( \frac{1}{\underline{\nu}T}\cdot\frac{K}{\underline{\nu}} + \frac{1}{ \underline{\nu}^2} \right)
  \limsup_{t \rightarrow \infty} \|f(t)\|_{V'}^2.
\end{align*}
Since $T \geqs c_{\rho}^2 / \underline{\nu} > 0$, 
we end up with~\eqref{eqn:time-energy3}.
\end{proof}

\begin{remark}
If one takes $\nu(t) \equiv c$, we find that 
$K=c_{\rho}^2 / \underline{\nu} $,
recovering the bounds for constant viscosity in both of the above estimates for time-varying viscosity
(see Appendix~\ref{app:apriori-constant-viscosity}).
\end{remark}

\section{Weak formulation and estimates for space-varying viscosity}
\label{sec:space-varying}

The remaining two sections of the notes are first steps to understanding the 
case of space-varying viscosity, and are both 
\emph{preliminary} and \emph{somewhat speculative}.
In this section, we attempt to develop a weak formulation that is appropriate 
for viscosity that can vary with the spatial location.
In the next section, and we establish some preliminary \emph{a priori} bounds 
for the space-varying case.
As was the case for the time-varying viscosity, we will make some basic
assumptions on the now space-varying viscosity:
\begin{align}
\label{eqn:space-nu-1}
0 < \underline{\nu} & \leqs \nu(x) \leqs \overline{\nu} < +\infty,
\quad \forall x \in \Omega,
\end{align}
where
\begin{align}
\label{eqn:space-nu-2}
\underline{\nu} = \inf_{x \in \Omega} \nu(x),
&
\qquad
\overline{\nu}  = \sup_{x \in \Omega} \nu(x).
\end{align}

Let us consider the effects of a space-varying viscosity on equations 
\eqref{eqn:weak} and \eqref{eqn:weak_ic}.
Our interest here is to develop a weak formulation analogous 
to~\eqref{eqn:weak}, but in which the viscosity is allowed to be 
space-varying, with its gradient is not necessarily zero.
Unlike with the time-dependent case, the NSE will 
now require an extra term $\nabla \nu \cdot \nabla u$ that we 
will call the \emph{viscosity-velocity divergence term}.
If $\nabla \nu \neq 0$, then we must consider $\nu A$ to be the modified 
Stokes operator.
We will assume that $\nu \in W$, where $W$ is an appropriate Banach space 
that will be determined later, such that $\nu A$ remains a bounded linear map,
where $A$ is the stokes operator as in the earlier discussion.
The term $\nu a(u, \eta)$ appearing in~\eqref{eqn:weak}
now becomes $a(\nu u, \eta)$.
We note that $\nu a(u, \eta)$ is bounded from below by $\nu |u|_{H^1(\Omega)}^2$ and $a(\nu u, \eta)$ is bounded from below by $|\nu u|_{H^1(\Omega)}^2$.

We begin with the NSE without consideration of viscosities $\lambda$ and $\nu$  as constants.
This can be written as follows from ~\cite{SeHa61}:
\begin{equation*}
\rho \frac{Du}{Dt}
= \rho f - \nabla p
+ \nabla (\lambda \nabla \cdot u)
+ \nabla \cdot (2 \nu D)
\end{equation*}
where $D_{rs} = 
\frac{1}{2}
(\frac{\partial u_r}{\partial x_s}
+ \frac{\partial u_s}{\partial x_r}
)$ is the symmetric component of the gradient of velocity, often referred to as the deformation tensor or rate-of-strain tensor, and $\frac{Du}{Dt} = \frac{\partial u}{\partial t}
+ u \cdot \nabla u$ is the material derivative.
The following hold:
\begin{align*}
\nabla \cdot (2 \nu D) &= \nabla (\nu \nabla \cdot u) 
     + \nabla \cdot (\nu \nabla u), 
\\
\nabla \cdot (\nu \nabla u)
&= [\nu, u] + \nu \nabla ^2 u,
\end{align*}
where $[\nu, u]_i = \nabla \nu \cdot \nabla u_i$ (we will use $\nabla \nu \cdot \nabla u$ to denote this term)
, and (assuming the fluid is incompressible) 
$$\nabla \cdot u = 0.$$
With this, we can write  
\begin{align}
\rho \frac{Du}{Dt}
&= \rho f - \nabla p
+ \nu \nabla ^2 u + \nabla \nu \cdot \nabla u
\label{eqn:ns-wk-space}
\\
u(0) &= u_0
\label{eqn:ns-wk-space-bc}
\end{align}

Multiplying both sides by a test function $\eta:\mathbb{R}^3 \rightarrow \mathbb{R}$, integrating over an appropriate domain $\Omega \subset \mathbb{R}^3$, we find that 
\begin{equation*}
\int_{\Omega} \rho \frac{\partial u}{\partial t} \cdot \eta
- \int_{\Omega} \nu \nabla^2 u \cdot \eta
+ \int_{\Omega} \rho (u \cdot \nabla) u \cdot \eta
+ \int_{\Omega} \nabla p \cdot \eta
- \int_{\Omega} (\nabla \nu \cdot \nabla u) \cdot \eta
= \int_{\Omega} \rho f \cdot \eta
\end{equation*}
We reverse-integrate by parts the diffusive term \(\mu \nabla^2 u \cdot \eta\) and pressure term \(\nabla p \cdot \eta\), and use the Divergence Theorem:
\begin{align*}
\int_{\Omega} \nabla p \cdot \eta
&= -\int_{\Omega} p \nabla \cdot \eta
   + \int_{\partial \Omega} p\eta \cdot \hat{n}
\\
-\int_{\Omega} \nabla^2 u \cdot \nu \eta
&= \int_{\Omega} \nabla u \cdot \nabla(\nu \eta)
   - \int_{\partial \Omega} \nu \frac{\partial u}{\partial \hat{n}} \cdot \eta
\\
&= \int_{\Omega} \nu \nabla u \cdot \nabla \eta
 + \int_{\Omega} \nabla u \nabla \nu \eta
 - \int_{\partial \Omega} \frac{\partial u}{\partial \hat{n}} \cdot \nu \eta
\end{align*}
Substituting back, the result is 
\begin{align*}
\int_{\Omega} \rho \frac{\partial u}{\partial t} \cdot \eta
&+ \int_{\Omega} \nu \nabla u \cdot \nabla \eta
+ \int_{\Omega} \nabla u \cdot \nabla \nu \eta
+ \int_{\Omega} \rho(u \cdot \nabla)u \cdot \eta
\\
&\quad
- \int_{\Omega} p \nabla \cdot \eta
- \int_{\Omega} (\nabla \nu \cdot \nabla u) \cdot \eta
\\
&= \int_{\Omega} \rho f \cdot \eta
+\int_{\partial \Omega} (\nu \frac{\partial u}{\partial \hat{n}} - p \hat{n}) \cdot \eta
\end{align*}
Choosing the test function $\eta$ so that $\eta = 0$ on $\partial \Omega$ removes the term involving the boundary integral.
The divergence constraint \(\nabla \cdot u = 0\) is now \(\int_{\Omega} q \nabla \cdot u = 0 \) \(\forall q \in Q = L^2(\Omega)\) .
With this, we can write the weak formulation of the NSE:
\\
Find \(u \in U = L^2((0,T);V)\) such that
\begin{align*}
\int_{\Omega} \rho \frac{\partial u}{\partial t} \cdot \eta
+ \int_{\Omega} \nu \nabla u \cdot \nabla \eta
& 
+ \int_{\Omega} \nabla u \cdot \nabla \nu \eta
+ \int_{\Omega} \rho(u \cdot \nabla)u \cdot \eta
\\
&
\qquad
- \int_{\Omega} p \nabla \cdot \eta
- \int_{\Omega} (\nabla \nu \cdot \nabla u) \cdot \eta
\\
& = \int_{\Omega} \rho f \cdot \eta,
\quad \forall \eta \in V = \mathbf{H}^1_0(\Omega),
\\
\int_{\Omega} q \nabla \cdot u &= 0,
    \quad \forall q \in Q = L^2(\Omega).
\end{align*}
Employing again the Leray orthogonal projector $P$ to incorporate
the divergence constraint into our functional framework, we have
the final weak formulation that allows for variable viscosity:
Given $f \in L^2((0,T);H)$, if $u \in L^2((0,T);V)$ satisfies
\begin{align}
   \label{eqn:weak-b}
\left( \frac{du}{dt}, \eta \right)
    + a(\nu u,\eta) + b(u,u,\eta) 
   - (a(\nu, u), \eta)
    &= (f,\eta), 
      ~~~\forall \eta \in V,
\\
u(0)&=u_0
\end{align}
then $u$ will be called a weak solution of the NSE 
with space-varying viscosity.


To go further with this analysis, we will need some type of a bound from below 
for the rather inconvenient term $a(\nu u, u)$, 
involving something more useful, such as some multiple of 
$\|u\|_{H^1(\Omega)}$.
Under suitable regularity assumptions on $\nu$ and $u$, we have the following.
\begin{proposition}
\label{prop:space-varying-coercivity}
Let $\nu \in H^1(\Omega)$ satisfy
\eqref{eqn:space-nu-1}--\eqref{eqn:space-nu-2},
and let $u \in V \cap H^2(\Omega)$.  Then
$$
a(\nu u, u) \geqs \underline{\nu} |u|_{H^1(\Omega)}^2.
$$
\end{proposition}
\begin{proof}
We first use the definition of the bilinear form $a(\cdot,\cdot)$ to write:
\begin{align*}
a(\nu u, u) 
& = (\nabla [\nu u], \nabla u)_{L^2(\Omega)}
= \int_{\Omega} \nabla [\nu u] \cdot \nabla u 
= \int_{\Omega} (\nabla \nu) u \nabla u 
+ \int_{\Omega} \nu (\nabla u)^2 
\\
& \geqs \int_{\Omega} (\nabla \nu) u \nabla u 
+ \underline{\nu} \int_{\Omega} (\nabla u)^2 
\\
& = \int_{\Omega} (\nabla \nu) u \nabla u 
+ \underline{\nu} |u|_{H^1(\Omega)}^2.
\end{align*}
We are done if we can show that \(\int_{\Omega} (\nabla \nu) u \nabla u  \geqs 0\).
To this end, we integrate by parts to find that 
\begin{align*}
\int_{\Omega} (\nabla \nu) u \nabla u 
& = \int_{\partial \Omega} \nu u \nabla u 
- \int_{\Omega} \nu u \nabla^2 u 
-\int_{\Omega} \nu (\nabla u)^2
\\
& \geqs \int_{\partial \Omega} \nu u \nabla u
- \int_{\Omega} \nu u \nabla^2 u
- \bar{\nu} |u|_{H^1(\Omega)}^2.
\end{align*}
Note that \(\int_{\partial \Omega} \nu u \nabla u\) vanishes due to the compact support of functions in $V$.
It remains to show that \(\int_{\Omega} \nu u \nabla^2 u + \bar{\nu} |u|_{H^1(\Omega)}^2 \leqs 0\); showing \(\int_{\Omega} u \nabla^2 u + |u|_{H^1(\Omega)}^2 \leqs 0\) will suffice.
Integrating by parts the integral term gives 
$$
\int_{\Omega} u \nabla^2 u
= \int_{\partial \Omega} u \nabla u
- \int_{\Omega} (\nabla u)^2
= -|u|_{H^1(\Omega)}^2.
$$
Thus, \(\int_{\Omega} u \nabla^2 u + |u|_{H^1(\Omega)}^2 \leqs 0\), which is what we were after.
\end{proof}

We now establish some basic preliminary \emph{a priori} bounds 
for the weak formulation of the simplified space-varying viscosity NSE model.

The first estimate gives a bound on the $L^2$-norm of a weak
solution to~\eqref{eqn:nv}--\eqref{eqn:nv_bc}.
\begin{proposition}[{{\bf $L^2$-Estimates, space-varying viscosity}}]
   \label{proposition:energy1-space}
Let $u \in L^2((0,T);V)$ be a weak solution of the Navier-Stokes 
equations~\eqref{eqn:ns-wk-space}--\eqref{eqn:ns-wk-space-bc}, 
with Lipschitz domain $\Omega \subset \mathbb{R}^d$, $d=2$ or $d=3$,
and assume that Proposition~\ref{prop:space-varying-coercivity} holds.
Then it holds that
\begin{equation}
   \label{eqn:energy1-space}
\limsup_{t\rightarrow \infty} \|u(t)\|_{H}^2 
\leqs  \frac{c_\rho^2}{2\underline{\nu}^2} \limsup_{t\rightarrow \infty} 
     \|f(t)\|_{V'}^2
     + \frac{c_\rho^2}{2\underline{\nu}^2}
     \limsup_{t \rightarrow \infty} \|\nabla \nu \cdot \nabla u(t)\|_{W}^2,
\end{equation}
where $c_\rho$ is the constant from the Poincar\'{e} inequality 
(Lemma~\ref{app:lemma:poincare} in Appendix~\ref{app:technical}),
and \(\underline{\nu} = \inf_{\Omega} \nu > 0\).
\end{proposition}
\begin{proof}
Beginning with \eqref{eqn:weak-b} for \(\eta = u\), using \eqref{eqn:dbydt}, noting that Theorem \ref{app:lemma:symmetries} guarantees \(b(u,u,u) = 0\), and under the assumption that Proposition~\ref{prop:space-varying-coercivity} holds, we can start with
$$
\frac{1}{2}
\frac{d}{dt}
\|u\|_H^2
+ \underline{\nu}
|u|_V^2
\leqs
\|f\|_{V'}
|u|_V
+ \|\nabla \nu \cdot \nabla u\|_W
|u|_V.
$$
Applying Young's Inequality leads to
\begin{align*}
\frac{d}{dt} \|u\|_H^2
+ 2 \underline{\nu} |u|_V^2
& \leqs \left( \sqrt{\underline{\nu}} \|f\|_{V'} \right)
\left( \frac{1}{\sqrt{\underline{\nu}}} |u|_V \right)
+ \left( \sqrt{\underline{\nu}} \|\nabla \nu \cdot \nabla u\|_{W} \right)
\left( \frac{1}{\sqrt{\underline{\nu}}} |u|_V \right)
\\
& \leqs \frac{\underline{\nu}}{2} \|f\|_{V'}^2
+ \frac{1}{2 \underline{\nu}} |u|_{V}^2
+ \frac{\underline{\nu}}{2} \|\nabla \nu \cdot \nabla u\|_W^2
+ \frac{1}{2 \underline{\nu}} |u|_V^2,
\end{align*}
which gives then 
\begin{equation}
\label{eqn:L2-space-equation}
\frac{d}{dt} \|u\|_H^2
+ \underline{\nu} |u|_V^2
\leqs
\frac{1}{2 \underline{\nu}}
\|f\|_{V'}^2
+ \frac{1}{2 \underline{\nu}}
\|\nabla \nu \cdot \nabla u\|_{W}^2.
\end{equation}
Employing the Poincar\'e inequality, we we end up with 
$$
\frac{d}{dt} \|u\|_H^2
+ \frac{\underline{\nu}}{c_\rho^2} \|u\|_H^2
\leqs \frac{1}{2 \underline{\nu}}
\|f\|_{V'}^2
+ \frac{1}{2 \underline{\nu}}
\|\nabla \nu \cdot \nabla u\|_{W}^2.
$$
This is a differential inequality for \(\|u(t)\|_H^2\), and by Gronwall's Inequality (Lemma \ref{app:lemma:gronwall-classical}) it holds that
\begin{align*}
\|u(t)\|_H^2
& \leqs \|u(r)\|_H^2 e^{-\int_r^t \underline{\nu}/c_\rho^2 d\tau}
+ \int_r^t \frac{1}{2 \underline{\nu}} \left( \|f(s)\|_{V'}^2 \| + \|\nabla \nu \cdot \nabla u(s) \|_W^2 \right)
e^{-\int_r^t  \underline{\nu} / c_\rho^2 d\tau} ds
\\
& = \|u(r)\|_H^2 e^{-\underline{\nu}(t-r)/c_\rho^2}
+ \int_r^t \frac{1}{2 \underline{\nu}} e^{-\underline{\nu}(t-s)/c_\rho^2} \left( \|f(s)\|_{V'}^2 \| + \|\nabla \nu \cdot \nabla u(s) \|_W^2 \right) ds
\\
& \leqs \|u(r)\|_H^2 e^{-\underline{\nu}(t-r)/c_\rho^2}
\\
& + \frac{1}{2 \underline{\nu}} \sup_{r \leqs \delta \leqs t}
\left( \|f(\delta)\|_{V'}^2 \| + \|\nabla \nu \cdot \nabla u(\delta) \|_W^2 \right) \int_r^t e^{\underline{\nu}(t-s)/c_\rho^2} ds
\\
& = \|u(r)\|_H^2 e^{-\underline{\nu}(t-r)/c_\rho^2} 
\\
& + \frac{1}{2 \underline{\nu}} \sup_{r \leqs \delta \leqs s}
\left( \|f(\delta)\|_{V'}^2 \| + \|\nabla \nu \cdot \nabla u(\delta) \|_W^2 \right) \frac{c_\rho}{\underline{\nu}^2} \left( e^0-e^{-\underline{\nu}(t-r)/c_\rho^2}\right),
\end{align*}
or more simply 
\begin{align*}
\|u(t)\|_H^2 \leqs \|u(r)\|_H^2
e^{\underline{\nu} (t-r)/c_\rho^2}
+ \frac{c_\rho^2}{2\underline{\nu}^2}
\sup_{r \leqs \delta \leqs t}
\left( \|f(\delta)\|_{V'}^2 \| + \|\nabla \nu \cdot \nabla u(\delta) \|_W^2 \right),
\end{align*}
which must hold for every \(r \in (0,t]\). Taking the \(\limsup_{t \rightarrow \infty}\) of both sides of the inequality leaves \eqref{eqn:energy1-space}
\end{proof}

The second estimate gives a bound on the time-averaged $H^1$-semi-norm 
of a weak solution to~\eqref{eqn:ns-wk-space}--\eqref{eqn:ns-wk-space-bc}.
\begin{proposition}[{{\bf Time-averaged $H^1$-Estimates, space-varying viscosity}}]
   \label{proposition:energy2-space}
Let $u \in L^2((0,T);V)$ be a weak solution of the Navier-Stokes 
equations~\eqref{eqn:ns-wk-space}--\eqref{eqn:ns-wk-space-bc}, 
with Lipschitz domain $\Omega \subset \mathbb{R}^d$, $d=2$ or $d=3$,
and assume that Proposition~\ref{prop:space-varying-coercivity} holds.
Then for every $T$ with $T \geqs c_\rho^2/\underline{\nu} > 0$ it holds that
\begin{equation}
\label{eqn:energy2-space}
\limsup_{t \rightarrow \infty} \frac{1}{T} \int_{t}^{t+T} 
   |u(\tau)|_{V}^2 d\tau 
   \leqs 
\frac{1}{\underline{\nu}^2}
\left(
  \limsup_{t \rightarrow \infty} 
   \|f(t)\|_{V'}^2 
      + \limsup_{t \rightarrow \infty}\|\nabla \nu \cdot \nabla u\|_{W}^2 \right),
\end{equation}
where $c_\rho$ is the constant from the Poincar\'{e} inequality 
(Lemma~\ref{app:lemma:poincare} in Appendix~\ref{app:technical}),
and \(\underline{\nu} = \inf_{\Omega} \nu > 0\).
\end{proposition}
\begin{proof}
  We can begin with \eqref{eqn:L2-space-equation} from the proof of
  Proposition~\ref{proposition:energy1-space}, which was
  $$
  \frac{d}{dt} \|u\|_H^2
  + \underline{\nu} |u|_V^2
  \leqs \frac{1}{2 \underline{\nu}} \|f\|_{V'}^2
  + \frac{1}{2 \underline{\nu}} \|\nabla \nu \cdot \nabla u\|_{W}^2.
  $$
  Integrating from \(t\) to \(t+t\) with \(T \geqs 0\) gives 
  $$
  \|u(t+T)\|_H^2 - \|u(t)\|_H^2
  + \underline{\nu} \int_t^{t+T} |u(\tau)|_V^2 d\tau
  \leqs 
  \frac{1}{2 \underline{\nu}} \int_t^{t+T} \|f(\tau)\|_{V'}^2
  + \frac{1}{2 \underline{\nu}} \int_t^{t+T} \|\nabla \nu \cdot \nabla u\|_{W}^2
  $$
  Dropping the positive first term on the left and bounding the integral on the right gives
$$
 \int_{t}^{t+T} |u(\tau)|_V^2
 \leqs \frac{1}{\underline{\nu}} \|u(t)\|_H^2
 + \frac{T}{2 \underline{\nu}^2} \sup_{t \leqs s \leqs t+T} \|f(s)\|_{V'}^2
 + \frac{T}{2 \underline{\nu}^2} \sup_{t \leqs s \leqs t+T} \|\nabla \nu \cdot \nabla u\|_{W}^2.
$$
Taking the \(\limsup_{t \rightarrow \infty}\) of both sides, and dividing by \(T\), gives 
\begin{align*}
\limsup_{t \rightarrow \infty} \frac{1}{T} \int_t^{t+T} |u(\tau)|_V^2
& \leqs \frac{1}{\underline{\nu}T} \limsup_{t \rightarrow \infty} \|u(t)\|_H^2
+ \frac{1}{2 \underline{\nu}^2} \limsup_{t \rightarrow \infty} \|f(t)\|_{V'}^2
\\
& + \frac{1}{2 \underline{\nu}^2} \limsup_{t \rightarrow \infty} \|\nabla \nu \cdot \nabla u(t)\|_{W}^2
\end{align*}
Using the estimate from Proposition \ref{proposition:energy1-space} gives then
\begin{align*}
\limsup_{t \rightarrow \infty} \frac{1}{T} \int_{t}^{t+T} |u(\tau)|_V^2
\leqs & \left( \frac{1}{\underline{\nu}T} \cdot \frac{c_\rho^2}{2 \underline{\nu}^2} + \frac{1}{2 \underline{\nu}^2} \right)
\\
& \left( \limsup_{t \rightarrow \infty} \|f(t)\|_{V'}^2
+ \limsup_{t \rightarrow \infty} \|\nabla \nu \cdot \nabla u(t)\|_W^2 \right)
\end{align*}
Since \(T \geqs c_\rho^2 / \underline{\nu}\), we end up with \eqref{eqn:energy2-space}
\end{proof}
\begin{remark}
To use these estimates with the determining projection
framework of~\S\ref{sec:approx-theory} (Theorem~\ref{thm:main-2d}),
it remains to determine appropriate function spaces for
$\nu$ so that terms involving $\nu$ and $\nabla \nu$ are well-defined
and compabible with both the weak formulation, 
the theory for weak solutions $u$, 
and any estimates we established above for determining projections.
Although $\nu$ spatially varies, it is taken here to be given as data,
and one can reverse-engineer any assumptions needed for 
e.g. $\nabla \nu \cdot \nabla u$,
or other terms involving $\nu$, to be well-defined.
Allowing for a more complicated class of variable viscosity,
such as viscosity that varies with the velocity,
would greatly complicate this discussion.
\end{remark}

\appendix
\section{Some technical tools}
\label{app:technical}

Here is a collection of some standard technical tools
that we use in the paper.

Young's inequality is used repeatedly throughout.
\begin{lemma}[{{\bf Young's Inequality}}]
   \label{app:lemma:young}
For $a,b\geqs 0$, $1<p,q<\infty$, $1/p+1/q=1$, it holds that
\begin{equation}
   \label{app:eqn:young}
ab \leqs \frac{a^p}{p} + \frac{b^q}{q}.
\end{equation}
\end{lemma}
\begin{proof}
See for example~\cite{KoFo70}.
\end{proof}

We use the Poincar\'{e} Inequality in several places;
in our setting, it takes the following form for both the classical
Sobolev space $H^1_0(\Omega)$ and the space of
vector-valued functions $\mathbf{H}^1_0(\Omega)$.
\begin{lemma}[{{\bf Poincar\'{e} Inequality}}]
   \label{app:lemma:poincare}
If $\Omega$ is bounded, then it holds that
\begin{equation}
   \label{app:eqn:poincare}
\| u \|_{L^2(\Omega)} \leqs c_\rho(\Omega) | u |_{H^1(\Omega)}, 
     ~~\forall u \in H^1_0(\Omega).
\end{equation}
\end{lemma}
\begin{proof}
For example see~\cite{Neca67}.
\end{proof}

In this paper, we use the notation \(H\) and \(V\) for \(L^2(\Omega)\) and \(H^1(\Omega)\), respectively. 
\\
The following \emph{a priori} bounds can be derived for the 
trilinear form $b(\cdot,\cdot,\cdot)$.
\begin{lemma}[{{\bf Trilinear Form Bounds}}]
  \label{app:lemma:lady}
If $\Omega \subset \mathbb{R}^d$, then the trilinear form 
$b(u,v,w)$ is bounded on $V \times V \times V$ as follows,
where $d=2$ or $d=3$ is the spatial dimension:
\begin{eqnarray}
d=2: & 
|b(u,v,w)| \leqs 2^{1/2} \|u\|_{L^2(\Omega)}^{1/2} |u|_{H^1(\Omega)}^{1/2} 
                      |v|_{H^1(\Omega)} 
                     \|w\|_{L^2(\Omega)}^{1/2} |w|_{H^1(\Omega)}^{1/2},
   \label{app:eqn:lady_2d} \\
d=3: &
|b(u,v,w)| \leqs 2 \|u\|_{L^2(\Omega)}^{1/4} |u|_{H^1(\Omega)}^{3/4} 
                      |v|_{H^1(\Omega)} 
                     \|w\|_{L^2(\Omega)}^{1/4} |w|_{H^1(\Omega)}^{3/4}.
   \label{app:lemma:lady_3d}
\end{eqnarray}
Moreover, from H\"older inequalities we have for $d=2$ or $d=3$:
\begin{equation}
   \label{app:eqn:l_infty}
|b(v,u,v)| \leqs \|\nabla u\|_{L^{\infty}(\Omega)} \|v\|_{L^2(\Omega)}^2.
\end{equation}
\end{lemma}
\begin{proof}
See~\cite{Lady69,Tema77,Tema83,CoFo88}.
\end{proof}

The following useful symmetries can be shown for the trilinear form.
\begin{lemma}[{{\bf Trilinear Form Symmetries}}]
  \label{app:lemma:symmetries}
If $\Omega \subset \mathbb{R}^d$, then the trilinear form 
$b(u,v,w)$ on $V \times V \times V$ has the following symmetries:
\begin{align*}
b(u,v,v) &= 0,
\\
b(u,v,w) &= -b(u,w,v),
\\
b(u-v,u,u-v) &= b(u,u,u-v) - b(v,v,u-v).
\end{align*}
\end{lemma}
\begin{proof}
See~\cite{Tema77,Tema83,CoFo88}.
\end{proof}

The classical Gronwall inequality is as follows.
\begin{lemma}[{{\bf Gronwall Inequality}}]
   \label{app:lemma:gronwall-classical}
If $\alpha(t)$ and $\beta(t)$ are real-valued and non-negative on 
$(0,\infty)$, and if the function $y(t)$ satisfies the following 
differential inequality:
$$
y^{\prime}(t) + \alpha(t) y(t) \leqs \beta(t), ~~\text{a.e.~on}~ (0,\infty),
$$
then $y(t)$ is bounded on $(0,\infty)$ by
\begin{align*}
y(t) & \leqs y(0) e^{-\int_0^t \alpha(\tau) d\tau}
       + \int_0^t \beta(s) e^{-\int_0^t \alpha(\tau) d\tau} ds.
\end{align*}
\end{lemma}
\begin{proof}
See for example~\cite{KoFo70}.
\end{proof}

The following generalized Gronwall inequality is used repeatedly
throughout Sections~\ref{sec:time-varying-estimates}
and~\ref{sec:space-varying} to obtain
\emph{a priori} estimates.
\begin{lemma}[{{\bf Generalized Gronwall Lemma}}]
   \label{app:lemma:gronwall-generalized}
Let $T>0$ be fixed, and let $\alpha(t)$ and $\beta(t)$ be locally integrable 
and real-valued on $(0,\infty)$, satisfying:
$$
\liminf_{t \rightarrow \infty}
    \frac{1}{T} \int_{t}^{t+T} \alpha(\tau) d\tau = m>0,
\ \ \ \ \ 
\limsup_{t \rightarrow \infty}
    \frac{1}{T} \int_{t}^{t+T} \alpha^-(\tau) d\tau = M<\infty,
$$
$$
\lim_{t \rightarrow \infty}
     \frac{1}{T} \int_{t}^{t+T} \beta^+(\tau) d\tau = 0,
$$
where $\alpha^-=\max\{-\alpha,0\}$ and $\beta^+=\max\{\beta,0\}$.
If $y(t)$ is an absolutely continuous non-negative function on
$(0,\infty)$, and $y(t)$ satisfies the following differential inequality:
$$
y'(t) + \alpha(t) y(t) \leqs \beta(t), \ \ \text{a.e.~on}~ (0,\infty),
$$
then $\lim_{t \rightarrow \infty} y(t) = 0$.
\end{lemma}
\begin{proof}
See~\cite{FMTT83,JoTi91}.
\end{proof}

\section{A priori estimates for constant viscosity}
\label{app:apriori-constant-viscosity}

Variations of the following two \emph{a priori} bounds on solutions
to the NSE can be found throughout the literature on the
Navier-Stokes equations; cf.~\cite{Tema77,Tema83,CoFo88}.
For example, Lemmas~\ref{app:lemma:energy1} and~\ref{app:lemma:energy2}
below (both from~\cite{Hols95b}) are simple generalizations to
$f \in V'$ of the bounds in e.g.~\cite{CoFo88}, 
presented there for $f \in H$.

The first estimate gives a bound on the $L^2$-norm of a weak
solution to~\eqref{eqn:nv}--\eqref{eqn:nv_bc}.
\begin{lemma}[{{\bf $L^2$-Estimates, constant viscosity}}]
   \label{app:lemma:energy1}
Let $u \in L^2((0,T);V)$ be a weak solution of the Navier-Stokes 
equations~\eqref{eqn:nv}--\eqref{eqn:nv_bc}, 
with Lipschitz domain $\Omega \subset \mathbb{R}^d$, $d=2$ or $d=3$.
It holds that
\begin{align}
   \label{eqn:energy1}
\limsup_{t\rightarrow \infty} \|u(t)\|_H^2
& \leqs  \frac{c_{\rho}^2}{\nu^2} \limsup_{t\rightarrow \infty} 
     \|f(t)\|_{V'}^2,
\end{align}
where $c_{\rho}$ is the constant from the Poincare inequality.
\end{lemma}
\begin{proof}
Beginning with equation \eqref{eqn:weak} for $v=u$, 
using~\eqref{eqn:dbydt}, and noting 
that Theorem~\ref{app:lemma:symmetries} in Appendix~\ref{app:technical}
ensures $b(u,u,u)=0$, we are left with
\begin{align*}
\frac{1}{2} \frac{d}{dt}\|u\|_H^2
   + \nu \| u \|_V^2
& \leqs \|f\|_{V'} \|u\|_V.
\end{align*}
Applying Young's inequality leads to
\begin{align*}
\frac{d}{dt}\|u\|_H^2
   + 2 \nu \| u \|_V^2 
 & \leqs \left( \sqrt{\frac{2}{\nu}} \|f\|_{V'} \right)
         \left( \sqrt{2 \nu} \|u\|_{V} \right)
  \leqs \frac{1}{\nu} \|f\|_{V'}^2
    + \nu \|u\|_{V}^2,
\end{align*}
which gives then
\begin{align}
\frac{d}{dt}\|u\|_H^2
  + \nu \|u\|_{V}^2 
& \leqs \frac{1}{\nu} \|f\|_{V'}^2.
   \label{eqn:weak-start1}
\end{align}
Employing the Poincar\'{e} inequality we end up with
\begin{align*}
\frac{d}{dt}\|u\|_H^2
  + \frac{\nu}{c_{\rho}^2} \|u\|_H^2 
& \leqs \frac{1}{\nu} \|f\|_{V'}^2.
\end{align*}
This is a differential inequality for $\|u(t)\|_H^2$,
and by Gronwall's Inequality (Lemma~\ref{app:lemma:gronwall-classical}) 
it holds that
\begin{align*}
\|u(t)\|_H^2 
& \leqs 
   \|u(r)\|_H^2 e^{-\int_r^t \nu/c_{\rho}^2 d\tau}
    + \int_r^t \frac{1}{\nu} \|f(s)\|_{V'}^2
        e^{-\int_{s}^t \nu/c_{\rho}^2 d\tau} ds
\\
&   =  \|u(r)\|_H^2 e^{-\nu(t-r)/c_{\rho}^2}
    + \int_r^t \frac{1}{\nu} e^{-\nu(t-s)/c_{\rho}^2}
        \|f(s)\|_{V'}^2 ds
\\
&   \leqs  \|u(r)\|_H^2 e^{-\nu(t-r)/c_{\rho}^2}
    + \frac{1}{\nu} \sup_{r\leqs \delta \leqs t} \|f(\delta)\|_{V'}^2 
\int_r^t e^{-\nu(t-s)/c_{\rho}^2} ds
\\
&    =  \|u(r)\|_H^2 e^{-\nu(t-r)/c_{\rho}^2}
    + \frac{1}{\nu} \sup_{r\leqs \delta \leqs t} \|f(\delta)\|_{V'}^2 
   \frac{c_{\rho}^2}{\nu} 
   \left( e^0 - e^{-\nu(t-r)/c_{\rho}^2} \right),
\end{align*}
or more simply
$$
\|u(t)\|_H^2
     \leqs  \|u(r)\|_H^2 e^{-\nu(t-r)/c_{\rho}^2}
+ \frac{c_{\rho}^2}{\nu^2} \sup_{r\leqs \delta \leqs t} \|f(\delta)\|_{V'}^2,
$$
which must hold for every $r \in (0,t]$.
Taking the $\limsup_{t\rightarrow \infty}$ of both
sides of the inequality leaves~\eqref{eqn:energy1}.
\end{proof}

The second estimate gives a bound on the time-averaged $H^1$-semi-norm 
of a weak solution to~\eqref{eqn:nv}--\eqref{eqn:nv_bc}.
\begin{lemma}[{{\bf Time-averaged $H^1$-Estimates, constant viscosity}}]
   \label{app:lemma:energy2}
Let $u \in L^2((0,T);V)$ be a weak solution of the Navier-Stokes 
equations~\eqref{eqn:nv}--\eqref{eqn:nv_bc}, 
with Lipschitz domain $\Omega \subset \mathbb{R}^d$, $d=2$ or $d=3$.
Then for every $T$ with 
$T \geqs c_{\rho}^2/\nu > 0$ 
it holds that
\begin{align}
   \label{eqn:energy2}
\limsup_{t \rightarrow \infty} \frac{1}{T} \int_{t}^{t+T} 
   \|u(\tau)\|_V^2 d\tau 
& \leqs \frac{2}{\nu^2}
      \limsup_{t \rightarrow \infty} \|f(t)\|_{V'}^2,
\end{align}
where $c_{\rho}$ is the constant from the Poincare inequality.
\end{lemma}
\begin{proof}
We begin with~\eqref{eqn:weak-start1}, which was
\begin{align*}
\frac{d}{dt}\|u\|_H^2
  + \nu \|u\|_{V}^2 
& \leqs \frac{1}{\nu} \|f\|_{V'}^2.
\end{align*}
Integrating from $t$ to $t+T$ with $T>0$ gives
\begin{align*}
\|u(t+T)\|_H^2 - \|u(t)\|_H^2
   + \nu \int_t^{t+T} \|u(\tau)\|_V^2 d\tau 
& \leqs \frac{1}{\nu} \int_t^{t+T} \|f(\tau)\|_{V'}^2 d\tau.
\end{align*}
Dropping the positive first term on the left and bounding the integral
on the right gives
\begin{align*}
\int_t^{t+T} \|u(\tau)\|_V^2 d\tau 
& \leqs \frac{1}{\nu} \|u(t)\|_H^2
  + \frac{T}{\nu^2} \sup_{t\leqs s \leqs t+T} \|f(s)\|_{V'}^2.
\end{align*}
Taking the $\limsup_{t \rightarrow \infty}$ of both sides, and
dividing by $T$, gives
\begin{align*}
\limsup_{t \rightarrow \infty}
\frac{1}{T}
\int_t^{t+T} \|u(\tau)\|_V^2 d\tau 
& \leqs \frac{1}{\nu T} 
  \limsup_{t \rightarrow \infty} \|u(t)\|_H^2
  + \frac{1}{\nu^2} 
        \limsup_{t\rightarrow \infty}
          \|f(t)\|_{V'}^2.
\end{align*}
Using the estimate from Lemma~\ref{app:lemma:energy1} gives then 
\begin{align*}
\limsup_{t \rightarrow \infty} \frac{1}{T}
\int_t^{t+T} \|u(\tau)\|_V^2 d\tau 
& \leqs \left( \frac{c_{\rho}^2}{\nu^3 T} + \frac{1}{ \nu^2} \right)
  \limsup_{t \rightarrow \infty} \|f(t)\|_{V'}^2.
\end{align*}
Since $T \geqs c_{\rho}^2 / \nu > 0$, 
we end up with~\eqref{eqn:energy2}.
\end{proof}

\section{Additional a priori estimates for time-varying viscosity}
\label{app:time-varying}

The following estimate gives a bound on the time-averaged product of viscosity and the $H^1$-semi-norm 
of the weak solution to~\eqref{eqn:nv}--\eqref{eqn:nv_bc}.
\begin{lemma}[{{\bf Time-averaged $H^1$-Estimates, time-varying viscosity}}]
   \label{lemma:energy2-time}
Let  \\ $u \in L^2((0,T);V)$ be a weak solution of the Navier-Stokes 
equations~\eqref{eqn:nv}--\eqref{eqn:nv_bc}, 
with Lipschitz domain $\Omega \subset \mathbb{R}^d$, $d=2$ or $d=3$.
Then for every $T$ with $T \geqs c_{\rho}^2/\underline{\nu} > 0$ it holds that

\begin{equation}
\label{eqn:energy2-time}
       \limsup_{t \rightarrow \infty}
       \frac{1}{T}
       \int_{t}^{t+T}
       \nu(\tau)
       \|u(\tau)\|_V^2
       d\tau
\leqs
       \frac{\bar{K}\underline{\nu} + c_{\rho}^2}{\underline{\nu} c_{\rho}^2}
       \limsup_{t \rightarrow \infty}
       \|f(t) \|_{V'}^2
\end{equation}
where \(\bar{K} = \limsup_{t \rightarrow \infty} \int_{0}^{t} e^{-\phi_{s}(t)) / c_\rho^2} ds\), \(\phi_{s}(t) = \int_{s}^t \nu(z)dz \), and \(c_{\rho}\) is the Poincar\'{e} constant.
\end{lemma}
\begin{proof}
We begin with~\eqref{eqn:time-weak-start1}, which was
\begin{align*}
\frac{d}{dt}\|u\|_H^2
  + \nu \|u\|_{V}^2 
& \leqs \frac{1}{\nu} \|f\|_{V'}^2.
\end{align*}
Integrating from $t$ to $t+T$ with $T>0$ gives
\begin{align*}
\|u(t+T)\|_H^2 - \|u(t)\|_H^2
   + \int_t^{t+T} \nu(\tau)\|u(\tau)\|_V^2 d\tau 
& \leqs  \int_t^{t+T} \frac{1}{\nu(\tau)}\|f(\tau)\|_{V'}^2 d\tau.
\end{align*}
Dropping the positive first term on the left and bounding the integral
on the right gives
\begin{align*}
\int_t^{t+T} \nu(\tau)\|u(\tau)\|_V^2 d\tau 
& \leqs \|u(t)\|_H^2
  + T \sup_{t\leqs s \leqs t+T} \frac{1}{\nu(s)}\|f(s)\|_{V'}^2.
\end{align*}
Taking the $\limsup_{t \rightarrow \infty}$ of both sides, and
dividing by $T$, gives
\begin{align*}
\limsup_{t \rightarrow \infty}
\frac{1}{T}
\int_t^{t+T} \nu(\tau)\|u(\tau)\|_V^2 d\tau 
& \leqs \frac{1}{ T} 
  \limsup_{t \rightarrow \infty} \|u(t)\|_H^2
  + \limsup_{t\rightarrow \infty}
          \frac{1}{\nu(t)}\|f(t)\|_{V'}^2.
\end{align*}
Using the estimate from Lemma~\ref{lemma:energy1-time} and bounding the right-most term gives then 
\begin{align*}
\limsup_{t \rightarrow \infty} \frac{1}{T}
\int_t^{t+T} \nu(\tau)\|u(\tau)\|_V^2 d\tau 
& \leqs \left( \frac{K}{\underline{\nu}T} + \frac{1}{ \underline{\nu}} \right)
  \limsup_{t \rightarrow \infty} \|f(t)\|_{V'}^2.
\end{align*}
Since $T \geqs c_{\rho}^2 / \underline{\nu} > 0$, 
we end up with~\eqref{eqn:energy2-time}.
\end{proof}

The following estimate is a variation of the other time-varying estimates from Lemmas \ref{lemma:energy3-time} and \ref{lemma:energy2-time}
and gives yet another slightly different bound on the time-averaged $H^1$-semi-norm 
of a weak solution to~\ref{eqn:nv}--\ref{eqn:nv_bc}.
\begin{proposition}[{{\bf Time-averaged $H^1$-Estimates, time-varying viscosity}}]
   \label{proposition:energy4-time}
 Let \\ $u \in L^2((0,T);V)$ be a weak solution of the Navier-Stokes 
equations~\eqref{eqn:nv}--\eqref{eqn:nv_bc}, 
with Lipschitz domain $\Omega \subset \mathbb{R}^d$, $d=2$ or $d=3$.
Then for every $T$ with $T \geqs c_\rho/\underline{\nu} > 0$ it holds that
\begin{equation*}
       \limsup_{t \rightarrow \infty}
       \frac{1}{T}
       \int_{t}^{t+T}
       \frac{1}{\nu(\tau)}
       |u(\tau)|_{H^1(\Omega)}^2
       d\tau
\leqs
       C
       \limsup_{t \rightarrow \infty}
       \|f(t) \|_{L^2(\Omega)}^2
\end{equation*}
        where \(C\) is dependent only on \(\bar{K} = \limsup_{t \rightarrow \infty} \int_{0}^{t} e^{-\phi_{s}(t)) / c_\rho^2} ds\), \(\phi_{s}(t) = \int_{s}^t \nu(z)dz \), and \(c_{\rho}\) the Poincar\'{e} constant.
\end{proposition}
\begin{proof}
\end{proof}
\bibliographystyle{abbrv}
\bibliography{m}
\vspace*{0.5cm}
\end{document}